\documentclass[]{article}
\usepackage{amsmath,amssymb,amsthm,mathrsfs,graphicx,subfigure,float,
caption,epstopdf,tabularx,bbm,color,bm,amsfonts,epic,booktabs,multirow}
\usepackage[colorlinks]{hyperref}

\usepackage{caption}
\allowdisplaybreaks[4]
\usepackage[toc,page]{appendix}

\usepackage{fullpage}
\usepackage{amssymb}
\usepackage{amsmath}
\usepackage{amsthm}
\usepackage[latin2]{inputenc}
\usepackage{amssymb}
\usepackage{paralist}
\usepackage{amsmath, calc}
\usepackage{amsfonts}
\usepackage{amsthm}
\usepackage{fullpage}
\usepackage{enumerate}
\usepackage{geometry}
\makeatletter

\newcommand{\Rmnum}[1]{\expandafter\@slowromancap\romannumeral#1@}
\makeatother

\numberwithin{equation}{section}
\newtheorem{theorem}{Theorem}[section]
\newtheorem{lemma}{Lemma}[section]

\newtheorem{remark}{Remark}[section]

\graphicspath{{Figs/}}

\begin{document}
\title{\Large On the integer sets with the same representation functions}
\date{}
\author{\large Kai-Jie Jiao,$^{1}$~ Csaba S\'andor,$^{2}$\footnote{Email: csandor@math.bme.hu. This author was supported by the OTKA Grant No. K129335.}~
Quan-Hui Yang$^{1}$\footnote{Email:~yangquanhui01@163.com.}~and~Jun-Yu Zhou$^{1}$}
\date{} \maketitle
 \vskip -3cm
\begin{center}

\vskip -1cm { \small 1. School of Mathematics and Statistics, Nanjing University of Information \\
Science and Technology, Nanjing 210044, China}
 \end{center}

 \begin{center}
{ \small 2. Institute of Mathematics, Budapest
University of Technology and Economics, and MTA-BME Lend\"{u}let Arithmetic Combinatorics Research Group, ELKH, H-1529 B.O. Box, Hungary}
 \end{center}

%
%
%
%
%
%

\begin{abstract}  Let $\mathbb{N}$ be the set of all nonnegative integers. For $S\subseteq \mathbb{N}$ and
$n\in \mathbb{N}$, let $R_S(n)$ denote the number of solutions of the equation
$n=s_1+s_2$, $s_1,s_2\in S$ and $s_1<s_2$. Let $A$ be the set of all nonnegative integers which contain an even number of digits $1$
in their binary representations and $B=\mathbb{N}\setminus A$.
Put $A_l=A\cap [0,2^l-1]$ and $B_l=B\cap [0,2^l-1]$.
In 2017, Kiss and S\'{a}ndor proved that, if
$C\cup D=[0,m]$, $0\in C$ and $C\cap D=\{r\}$, then $R_C(n)=R_D(n)$ 
for every positive integer $n$ if and only if there exists an integer 
$l\ge 1$ such that $r=2^{2l}-1$, $m=2^{2l+1}-2$,
$C=A_{2l}\cup (2^{2l}-1+B_{2l})$ and $D=B_{2l}\cup (2^{2l}-1+A_{2l})$.
This solved a problem of Chen and Lev. In this paper, we prove that,
if $C \cup D=[0, m]\setminus \{r\}$ with $0<r<m$, $C \cap D=\emptyset$ and $0 \in C$, then $R_{C}(n)=R_{D}(n)$ for any nonnegative integer $n$ if and only if there exists an integer $l \geq 2$ such that $m=2^{l}$, $r=2^{l-1}$, $C=A_{l-1} \cup\left(2^{l-1}+1+B_{l-1}\right)$ and $D=B_{l-1} \cup\left(2^{l-1}+1+A_{l-1}\right)$.

{\it Keywords:} S\'{a}rkozy's problem, Thue-Morse sequence, representation function.
\end{abstract}

\section{Introduction}

Let $\mathbb{N}$ be the set of all nonnegative integers. For $S\subseteq \mathbb{N}$ and
$n\in \mathbb{N}$, let $R_S(n)$ denote the number of solutions of the equation
$n=s_1+s_2$, $s_1,s_2\in S$ and $s_1<s_2$. S\'{a}rk\"{o}zy asked whether there exist two
sets $U$ and $V$ of nonnegative integers with infinite symmetric difference and $R_U(n)=R_V(n)$ for all large enough integers $n$. Dombi \cite{dombi} answered this problem affirmatively by using the Thue-Morse sequence. Later Lev \cite{lev}, S\'andor \cite{csandor}, and Tang \cite{tang08} gave a simple proof of this result respectively. For related results, one can refer to \cite{chenwang}, \cite{yangchen}-\cite{yutang}.

Let $A$ be the set of all nonnegative integers which contain an even number of digits $1$
in their binary representations and $B=\mathbb{N}\setminus A$.
For any positive integer $l$, let $A_l=A\cap [0,2^l-1]$ and $B_l=B\cap [0,2^l-1]$.
In 2017, Kiss and S\'{a}ndor \cite{kiss} proved that if
$C$ and $D$ are two sets of nonnegative integers such that $C\cup D=[0, m]$, $C\cap D=\emptyset$ and $0\in C$, then $R_C(n)=R_D(n)$ for every positive integer $n$ if and only if there exists a positive integer $l$ such that $C=A_l$ and $D=B_l$. In 2016, Tang \cite{tang16} proved
that for any given integer $m$, there do not exist subsets $C,D\subseteq \mathbb{N}$
with $\mathbb{N}=C\cup D$ and $C\cap D=\{km:~k\in \mathbb{N}\}$ such that
$R_C(n)=R_D(n)$ for all sufficiently large integers $n$.

For any integer $r$ and $m$, we define $r+m\mathbb{N}=\{r+ma:~a\in \mathbb{N}\}$. Chen and Lev \cite{chenlev} proved that, for any given positive integer $l$, there exist two
subsets $C,D\subseteq \mathbb{N}$ with $C\cup D=\mathbb{N}$ and $C\cap D=2^{2l}-1+(2^{2l+1}-1)\mathbb{N}$ such that $R_C(n)=R_D(n)$
for every positive integer $n$. They also posed the following two problems.

{\bf Problem 1.} Given that $R_C(n)=R_D(n)$ for every positive integer $n$,
$C\cup D=\mathbb{N}$, and $C\cap D=r+m\mathbb{N}$ with integers $r\ge 0$ and
$m\ge 2$, must there exist an integer $l\ge 1$ such that $r=2^{2l}-1$, $m=2^{2l+1}-1$?

{\bf Problem 2.} Given that $R_C(n)=R_D(n)$ for every positive integer $n$,
$C\cup D=[0,m]$, and $C\cap D=\{r\}$ with integers $r\ge 0$ and
$m\ge 2$, must there exist an integer $l\ge 1$ such that $r=2^{2l}-1$, $m=2^{2l+1}-2$,
$C=A_{2l}\cup (2^{2l}-1+B_{2l})$ and $D=B_{2l}\cup (2^{2l}-1+A_{2l})$?

In 2017, Kiss and S\'{a}ndor solved Problem 2 affirmatively. Later, Li and Tang \cite{litang},
Chen, Tang and Yang \cite{chentang} solved Problem 1 under the assumption $0\le r<m$.
Recently, Chen and Chen \cite{chenchen} solved Problem 1 affirmatively.

In this paper, we consider two sets $C$ and $D$ satisfying $C\cup D=[0,m]\setminus \{r\}$ and $C \cap D=\emptyset$. We prove the following result.

\begin{theorem}\label{thm1} Let $C$ and $D$ be two sets of nonnegative integers such that
$C \cup D=[0, m]\setminus \{r\}$ with $0<r<m$, $C \cap D=\emptyset$ and $0 \in C$.
Then for any nonnegative integer $n$, $R_{C}(n)=R_{D}(n)$ if and only if there exists an integer $l \geq 2$ such that $m=2^{l}$, $r=2^{l-1}$, $C=A_{l-1} \cup\left(2^{l-1}+1+B_{l-1}\right)$ and $D=B_{l-1} \cup\left(2^{l-1}+1+A_{l-1}\right)$.
\end{theorem}

\section{Preliminary Lemmas}

\begin{lemma}\label{lem1} For positive integers $0<r_1<r_2<\cdots<r_s\le m$,
there exist at most a pair of sets $(C,D)$ such that $C\cup D=[0,m]\setminus \{r_1,r_2,\ldots,r_s\}$, $0\in C$, $C\cap D=\emptyset$ and $R_C(k)=R_D(k)$ for all $k\le m$.
\end{lemma}

\begin{proof} Suppose that there exist two pair of sets $\left(C_{1}, D_{1}\right),\left(C_{2}, D_{2}\right)$ satisfying the condition of Lemma \ref{lem1}.
Let $v$ be the minimal nonnegative integer such that $\chi_{C_{1}}(v) \neq \chi_{C_{2}}(v)$.
Then $C_{1} \cap[0, v-1]=C_{2} \cap[0, v-1]$ and $R_{C_1}(v)\not=R_{C_2}(v)$ since $0\in C_1,C_2$.
By $R_{C_1}(v)=R_{D_1}(v)$ and $R_{C_2}(v)=R_{D_2}(v)$,
it suffices to prove that $R_{D_1}(v)=R_{D_2}(v)$ for a contradiction.

Since $0 \in C$ and $C \cap D=\emptyset$, it follows that $0\notin D$, and so $$R_{D_{1}}(v)=\left|\left\{\left(d, d^{\prime}\right): d<d^{\prime}<v, d, d^{\prime} \in D_{1}, d+d^{\prime}=v\right\}\right|.$$

Let $t$ be an integer with $0\le t\le s$ such that $r_{t}<v<r_{t+1}$. Here we define
$r_{0}=0$, $r_{s+1}=+\infty$.

Since
$$[0, v-1] \setminus\left\{r_{1},r_2, \ldots, r_{t}\right\}=\left(C_{1} \cap[0, v-1]\right) \cup\left(D_{1} \cap[0, v-1]\right),$$
it follows that
$$D_{1} \cap[0, v-1]=\left([0, v-1]\setminus\left\{r_{1}, r_2,\ldots, r_{t}\right\}\right)
\setminus \left(C_{1} \cap[0, v-1]\right)=[0, v-1] \setminus\left(\left(C_{1} \cap[0, v-1]\right) \cup\left\{r_{1},r_2, \ldots, r_{s}\right\}\right).$$
Similarly,
$$D_{2} \cap[0, v-1]=[0, v-1]\setminus\left(\left(C_{2} \cap[0, v-1]\right) \cup\left\{r_{1},r_2, \ldots, r_{t}\right\}\right).$$
Hence $D_{1} \cap[0, v-1]=D_{2} \cap[0, v-1]$, and so $R_{D_{1}}(v)=R_{D_{2}}(v)$.
\end{proof}

\begin{remark} If $\{r_1,r_2,\ldots,r_s\}=\emptyset$, then Lemma \ref{lem1} also holds.
\end{remark}

\begin{lemma}\label{lem2} Suppose that $C,D$ are two distinct sets satisfying ${C} \cup {D}=[0, {m}]\setminus \left\{{r}_{1}, \ldots, {r}_{s}\right\}$, ${C} \cap {D}=\emptyset$ and
${R}_{C}({n})={R}_{D}({n})$ for any positive integer $n$. If ${C}^{\prime}={m}-{C}$, ${D}^{\prime}={m}-{D}$, then ${C}^{\prime} \cup {D}^{\prime}=[0, {m}] \setminus \left\{{m}-{r}_{1}, \ldots, {m}-{r}_{s}\right\}, {C}^{\prime} \cap {D}^{\prime}=\emptyset$, and ${R}_{{C}^{\prime}}({n})={R}_{{D}^{\prime}}({n})$ for all nonnegative integers $n$.
\end{lemma}

\begin{proof} For $k=0,1,\ldots,2m$,
\begin{eqnarray*}
R_{C^{\prime}}(k)&=&\left|\left\{\left(c, c^{\prime}\right): c<c^{\prime}, c, c^{\prime} \in C^{\prime}, c+c^{\prime}=k\right\}\right| \\
&=&\left|\left\{\left(c, c^{\prime}\right): c<c^{\prime}, m-c, m-c^{\prime} \in C, c+c^{\prime}=k\right\}\right| \\
&=&\left|\left\{\left(m-c, m-c^{\prime}\right): c<c^{\prime}, m-c, m-c^{\prime} \in C, 2 m-\left(c+c^{\prime}\right)=2 m-k\right\}\right| \\
&=&R_{C}(2 m-k).
\end{eqnarray*}
Similarly, $R_{D^{\prime}}(k)=R_{D}(2 m-k)$. Since $R_{C}(2 m-k)=R_{D}(2 m-k)$, it follows that $R_{C^{\prime}}(k)=R_{D^{\prime}}(k)$.
\end{proof}

\begin{lemma}\label{lem3}(See \cite[Claim 3]{kiss}) If for some positive integer $M$, the integers $M-1, M-2, M-4,
M-8,\ldots, M-2^{\left\lceil\log _{2} M\right\rceil-1}$ are all contained in the set $A$, then $\left\lceil\log _{2} M\right\rceil$ is odd and $M=2^{\left\lceil\log _{2} M\right\rceil}-1$.
\end{lemma}

\begin{lemma}\label{lem4}(See \cite[Claim 4]{kiss}) If for some positive integer $M$, the integers $M-1, M-2, M-4,
M-8,\ldots, M-2^{\left\lceil\log _{2} M\right\rceil-1}$ are all contained in the set $B$, then $\left\lceil\log _{2} M\right\rceil$ is even and $M=2^{\left\lceil\log _{2} M\right\rceil}-1$.
\end{lemma}

\begin{lemma}\label{lem5} If for some even positive integer $M$, the integers $M-2, M-4, \ldots, M-2^{\left\lceil\log _{2} M\right\rceil-1}$ are all contained in the set $A$, then $M=2^{\left\lceil\log _{2} M\right\rceil}-2$.
\end{lemma}

\begin{proof} If for some even positive integer $M$, the integers $M-2, M-4, \ldots, M-2^{\left\lceil\log _{2} M\right\rceil-1}$ are all contained in the set $A$, then $\frac{M}{2}-1,\frac{M}{2}-2,\ldots,
\frac{M}{2}-2^{\left\lceil\log _{2} \frac{M}{2}\right\rceil-1}$ are all contained in the set $A$. By Lemma
\ref{lem3}, it follows that $\frac{M}{2}=2^{\left\lceil\log _{2} \frac{M}{2}\right\rceil}-1$, and so
$M=2^{\left\lceil\log _{2} M\right\rceil}-2$.
\end{proof}

\begin{lemma}\label{lem6}(See \cite[Theorem 3]{kiss}) Let $C$ and $D$ be sets of nonnegative integers such that $C\cup D=[0,m]$
and $C\cap D=\emptyset$, $0\in C$. Then $R_C(n)=R_D(n)$ for every positive integer $n$ if and only if
there exists an nonnegative integer $l$ such that $C=A_l$ and $D=B_l$.
\end{lemma}

\begin{lemma}\label{lem7}(See \cite[Corollary 1]{kiss}) If $C=A\cap [0,m]$ and $D=B\cap [0,m]$, where $m$ is a positive integer
not of the form $2^l-1$, then there exists a positive integer $m<n<2m$ such that $R_C(n)\not=R_D(n)$.
\end{lemma}

\section{Proofs}

\begin{proof}[Proof of Theorem \ref{thm1}.] (Sufficiency). Suppose that there exist a positive integer $l \geq 2$ such that $C=A_{l-1} \cup\left(2^{l-1}+1+B_{l-1}\right)$ and $D=B_{l-1} \cup\left(2^{l-1}+1+A_{l-1}\right)$.
Now we prove that $R_{C}(n)=R_{D}(n)$ for all nonnegative integers $n$. Clearly
$$
\begin{aligned}
R_{C}(n) &=\left|\left\{\left(c, c^{\prime}\right): c<c^{\prime}, c, c^{\prime} \in A_{l-1}, c+c^{\prime}=n\right\}\right|+\left|\left\{\left(c, c^{\prime}\right): c \in A_{l-1}, c^{\prime} \in 2^{l-1}+1+B_{l-1}, c+c^{\prime}=n\right\}\right| \\
&+\left|\left\{\left(c, c^{\prime}\right): c<c^{\prime}, c, c^{\prime} \in 2^{l-1}+1+B_{l-1}, c+c^{\prime}=n\right\}\right| \\
&=R_{A_{l-1}}(n)+|\{(c, c^{\prime}): c \in A_{l-1}, c^{\prime} \in B_{l-1}, c+c^{\prime}=n-(2^{l-1}+1) \}|+R_{B_{l-1}}(n-2(2^{l-1}+1)).
\end{aligned}
$$
Similarly
$$
\begin{gathered}
R_{D}(n)=R_{B_{l-1}}(n)+|\{(d, d^{\prime}): d \in B_{l-1}, d^{\prime} \in A_{l-1}, d+d^{\prime}=n-(2^{l-1}+1)\}|+R_{A_{l-1}}(n-2(2^{l-1}+1)).
\end{gathered}
$$

By Lemma \ref{lem1}, for any positive integer $m$, $R_{A_{l-1}}(m)=R_{B_{l-1}}(m)$. Hence $R_{C}(n)=R_{D}(n)$.

(Necessity). Suppose that $C \cup D=[0, m]\setminus \{r\}$, $C \cap D=\emptyset$, and for any positive
integer $n$, $R_{C}(n)=R_{D}(n)$. Let $C^{\prime}=m-C$, $D^{\prime}=m-D$. Then by
Lemma \ref{lem2}, we have $C^{\prime} \cup D^{\prime}=[0, m]\setminus \{m-r\}$, $C^{\prime} \cap D^{\prime}=\emptyset$ and $R_{C^{\prime}}(n)=R_{D^{\prime}}(n)$ for any positive integer.
Hence we assume that $r \leq \frac{m}{2}$ without loss of generality.

Let $p_{C}(x)=\sum_{i=0}^{m} \chi_{C}(i) x^{i}$, $p_{D}(x)=\sum_{i=0}^{m} \chi_{D}(i) x^{i}$.
Since $C \cup D=$ $[0, m]\setminus \{r\}$, $C \cap D=\emptyset$, it follows that
$$p_{C}(x)+p_{D}(x)=\frac{1-x^{m+1}}{1-x}-x^{r}.$$
Hence $p_{D}(x)=\frac{1-x^{m+1}}{1-x}-p_{C}(x)-x^{r}$.
Since $R_{C}(n)=R_{D}(n)$ for any positive integer $n$, we have
$$\sum_{n=0}^{\infty} R_{C}(n) x^{n}=\sum_{n=0}^{\infty} R_{D}(n) x^{n}.$$
Hence
\begin{eqnarray*}&&\frac{1}{2} p_{C}(x)^{2}-\frac{1}{2} p_{C}(x^{2})=\frac{1}{2}\left(\sum_{i \in C} x^{i}\right)^{2}-\frac{1}{2} \sum_{i \in C} x^{2 i}=\sum_{\substack{i_{1},i_2 \in C\\ i_{1}<i_{2}}} x^{i_{1}} \cdot x^{i_{2}}\\
&=&\sum_{\substack{i_{1},i_2 \in C\\ i_{1}<i_{2}}} x^{i_{1}+i_{2}}=\sum_{n=0}^{\infty} R_{C}(n) x^{n}=\sum_{n=0}^{\infty} R_{D}(n) x^{n}
=\frac{1}{2} p_{D}(x)^{2}-\frac{1}{2} p_{D}(x^{2}),
\end{eqnarray*}
and then
\begin{eqnarray*}p_{C}(x)^{2}-p_{C}\left(x^{2}\right)&=&\left(\frac{1-x^{m+1}}{1-x}-p_{C}(x)-x^{r}\right)^{2}-\left(\frac{1-x^{2 m+2}}{1-x^{2}}-p_{C}\left(x^{2}\right)-x^{2 r}\right)\\
&=&\left(\frac{1-x^{m+1}}{1-x}\right)^{2}+p_{C}(x)^{2}+x^{2 r}-2 \frac{1-x^{m+1}}{1-x} \cdot p_{C}(x)-2 \frac{1-x^{m+1}}{1-x} \cdot x^{r}\\
&&+2 p_{C}(x)x^{r}-\frac{1-x^{2 m+2}}{1-x^{2}}+p_{C}\left(x^{2}\right)+x^{2 r}.
\end{eqnarray*}
Therefore,
\begin{eqnarray}\label{eq1.0}2 p_{C}\left(x^{2}\right)=\frac{1-x^{2 m+2}}{1-x^{2}}+2 p_{C}(x) \frac{1-x^{m+1}}{1-x}-\left(\frac{1-x^{m+1}}{1-x}\right)^{2}+2 x^{r} \frac{1-x^{m+1}}{1-x}-2 p_{C}(x) x^{r}-2 x^{2 r}.\end{eqnarray}

We first check the necessity for $1\le r\le 5$.

For $r=1$, if there is a suitable decomposition $C\cup D=[0,m]\setminus \{ 1 \}$ and $m\ge 4$,
then $0\in C$ and $2,3,4\in D$. Let $E=m-C$, $F=m-D$. Then
$$(E\cap [0,m-2])\cap (F\cap [0,m-2])=\emptyset,\quad (E\cap [0,m-2])\cup (F\cap [0,m-2])=[0,m-2].$$
By Lemma \ref{lem2}, $R_E(n)=R_F(n)$ for all integers $n\le m-2$. By Lemma \ref{lem1}, $F\cap [0,m-2]=A\cap [0,m-2]$ or $B\cap [0,m-2]$, but this is a contradiction, because $F\cap [0,m-2]$ contains three consecutive integers $m-4$, $m-3$ and $m-2$. Hence the only solution is $m=2$.

For $r=2,3,4,5$, the proof is similar. Here we give the details in the case $r=4$.
Since $0\in C$, we have $1\in D$, otherwise $R_C(1)\not=R_D(1)$. Similarly we can obtain
$2\in D$, $3\in C$, $5\in D$, $6\in C$, $7\in C$, $8\in D$,
$9\in D$, $10\in C$, $11\in D$, $12\in C$, $13\in D$, and we can not put $14$ in both $C$ and $D$.
Hence if $m\ge 14$, there is no $C$ and $D$. For $m<13$, we can check only $m=8$ holds and
$C=\{0,3,6,7\}$, $D=\{1,2,5,8\}$.

Now we assume $r\ge 6$. First we shall prove that $r$ is even.
Suppose that $r$ is odd. Let $k$ be an even integer such that $r<k<2r$.
Then the coefficient of $x^k$ in \eqref{eq1.0} is
\begin{eqnarray}\label{eq1.1}
2 \chi_{C}\left(\frac{k}{2}\right)=1+2 \sum_{i \leq k} \chi_{C}(i)-(k+1)+2-2 \chi_{C}(k-r),
\end{eqnarray}
and the coefficient of $x^{k-1}$ in \eqref{eq1.0} is
\begin{eqnarray}\label{eq1.2}0=0+2 \sum_{i \leq k-1} \chi_{C}(i)-k+2-2 \chi_{C}(k-1-r).\end{eqnarray}
Calculating $(\eqref{eq1.1}-\eqref{eq1.2})\times \frac{1}{2}$, we have
\begin{eqnarray}\label{eq1.3}
\chi_{C}\left(\frac{k}{2}\right)=\chi_{C}(k)-\chi_{C}(k-r)+\chi_{C}(k-1-r).
\end{eqnarray}
By Lemma \ref{lem1}, we have
\begin{eqnarray}\label{eq1.4} C \cap[0, r-1]=A \cap[0, r-1],\end{eqnarray}
\begin{eqnarray}\label{eq1.5} D \cap[0, r-1]=B \cap[0, r-1].\end{eqnarray}

Since $k-1-r$ is even, $k-r<r$, by \eqref{eq1.4} and the definition of $A$,
we have
$$\chi_{C}(k-r)+\chi_{C}(k-1-r)=1.$$

If $\chi_{C}(k-r)=0, \chi_{C}(k-1-r)=1$, then by \eqref{eq1.3}, we have $\chi_{C}(k)=0$ and $\chi_{C}\left(\frac{k}{2}\right)=1$.

If $\chi_{C}(k-r)=1, \chi_{C}(k-1-r)=0$, then by \eqref{eq1.3}, we have $\chi_{C}(k)=1$ and $\chi_{C}\left(\frac{k}{2}\right)=0$.

Hence $\chi_{C}(k-1-r)=\chi_{C}\left(\frac{k}{2}\right)$.

Let $k=2 r-2^{i+1}+2$, where $i \geq 1$. Then $\frac{k}{2}=r+1-2^{i}$, $k-1-r=r+1-$ $2^{i+1}$.
Hence $\chi_{C}\left(r+1-2^{i}\right)=\chi_{C}\left(r+1-2^{i+1}\right)$. Therefore,
$$
\chi_{C}(r+1-2)=\chi_{C}(r+1-4)=\cdots=\chi_{C}\left(r+1-2^{\left\lceil\log _{2}(r+1) \rceil-1\right.}\right),
$$
and so
$$
\chi_{A}(r+1-2)=\chi_{A}(r+1-4)=\cdots=\chi_{A}\left(r+1-2^{\left\lceil\log _{2}(r+1)\rceil-1\right.}\right).
$$
By Lemma \ref{lem5}, it follows that $r+1=2^{u}-2$ and $\chi_{A}(r+1)=\chi_{A}(r)$.

Next we shall prove $r\in B$.
Suppose that $r \in A$. Since $0 \in C$, $r \notin C$, by \eqref{eq1.4}, it follows that
\begin{eqnarray}\label{eq1.6}
R_{C}(r)=R_{C \cap[0, r-1]}(r)=R_{A \cap[0, r-1]}(r)=R_{A \cap[0, r]}(r)-1.
\end{eqnarray}
Similarly, by \eqref{eq1.5}, it follows that
\begin{eqnarray}\label{eq1.7}
R_{D}(r)=R_{D \cap[0, r-1]}(r)=R_{B \cap[0, r-1]}(r)=R_{B \cap[0, r]}(r).
\end{eqnarray}
Noting that
$$R_{A \cap[0, r]}(r)=R_{A}(r)=R_{B}(r)=R_{B \cap[0, r]}(r),$$
we have $R_{C}(r) \neq R_{D}(r)$, a contradiction.
Hence $r \in B$ and $r+1 \in B$, and so $r \notin A$.

By \eqref{eq1.4} and $r \notin C$, we have $C \cap[0, r]=A \cap[0, r]$. Hence
$$
R_{C}(r+1)=R_{C \cap[0, r]}(r+1)+\chi_{C}(r+1)=R_{A \cap[0, r]}(r+1)+\chi_{C}(r+1).
$$

By \eqref{eq1.5} and $1, r \in B, r \notin D$, we have
$$R_{D}(r+1)=R_{B \cap[0, r]}(r+1)-1.$$

Noting that $r,r+1\in B$, we have $$R_{A \cap[0, r]}(r+1)=R_{A \cap[0, r+1]}(r+1),
\quad R_{B \cap[0, r]}(r+1)=R_{B \cap[0, r+1]}(r+1).$$ Since
$$R_{A \cap[0, r+1]}(r+1)=R_{A}(r+1)=R_{B}(r+1)=R_{B \cap[0, r+1]}(r+1),$$
it follows that $R_{C}(r+1) \neq R_{D}(r+1)$, a contradiction.

Therefore, $r$ is not odd.

Suppose that $r$ is even. Take an even integer $k$ with $r \leq k \leq 2 r$.
Then the coefficient of $x^{k-2}$ in \eqref{eq1.0} is
\begin{eqnarray}\label{eq1.9}
2 \chi_{C}\left(\frac{k-2}{2}\right)=1+2 \sum_{i \leq k-2} \chi_{C}(i)-(k-1)+2-2 \chi_{C}(k-2-r),
\end{eqnarray}
and the coefficient of $x^{k-1}$ in \eqref{eq1.0} is
\begin{eqnarray}\label{eq1.10}
0=0+2 \sum_{i \leq k-1} \chi_{C}(i)-k+2-2 \chi_{C}(k-1-r).
\end{eqnarray}
Calculating $(\eqref{eq1.9}-\eqref{eq1.10})\times \frac{1}{2}$, we obtain $$\chi_{C}\left(\frac{k-2}{2}\right)=1-\chi_{C}(k-1)-\chi_{C}(k-2-r)+\chi_{C}(k-1-r).$$

By Lemma \ref{lem2}, $k-1-r<r$ and $k-2-r$ is even. By \eqref{eq1.4} and the definition of $A$,
$$\chi_{C}(k-2-r)+\chi_{C}(k-1-r)=1.$$

If $\chi_{C}(k-2-r)=0$ and $\chi_{C}(k-1-r)=1$, then $\chi_{C}(k-1)=1, \chi_{C}\left(\frac{k-2}{2}\right)=1$.

If $\chi_{C}(k-2-r)=1$, $\chi_{C}(k-1-r)=0$, then $\chi_{C}(k-1)=0, \chi_{C}\left(\frac{k-2}{2}\right)=0$.

By the above two cases, we have $\chi_{C}(k-1-r)=\chi_{C}\left(\frac{k-2}{2}\right)$.
Take $$k=2 r-2^{i+1}, i=0,1,\ldots,\lceil \log_2(r-1)\rceil-2,$$ then $\frac{k-2}{2}=r-1-2^{i}, k-1-r=r-1-2^{i+1}$, and so $\chi_{C}\left(r-1-2^{i}\right)=\chi_{C}\left(r-1-2^{i+1}\right)$. Hence
$$
\chi_{C}(r-1-1)=\chi_{C}(r-1-2)=\cdots=\chi_{C}\big(r-1-2^{\left\lceil \log _{2}(r-1)\right\rceil -1}\big).
$$
By lemma \ref{lem4}, we have $r-1=2^{\left\lceil\log _{2}(r-1)\right\rceil}-1$. Write $l=\left\lceil\log _{2}(r-1)\right\rceil+1$. Then $r=2^{l-1}$. Noting that $r\le \frac{m}{2}$, then $[0,2r]\subseteq [0,m]$.
By Lemma 1 and the sufficiency of Theorem 1, it follows that
$$
C \cap[0,2 r]=A_{l-1} \cup (2^{l-1}+1+B_{l-1} ), \quad D \cap[0,2 r]=B_{l-1} \cup (2^{l-1}+1+A_{l-1}).
$$

Next we shall prove $m=2r$. First, we prove $m<3 \cdot 2^{l-1}=2^{l-1}+2^{l}$.
Suppose that $m \geq 3 \cdot 2^{l-1}$. First we shall prove the following structure of $C$ and $D$:
$$
\begin{aligned}
&C \cap\left[0,3 \cdot 2^{l-1}\right]=A_{l-1} \cup\left(2^{l-1}+1+B_{l-1}\right) \cup\left(2^{l}+1+B_{l-1}\right),\\
&D \cap\left[0,3 \cdot 2^{l-1}\right]=B_{l-1} \cup\left(2^{l-1}+1+A_{l-1}\right) \cup\left(2^{l}+1+A_{l-1}\right).
\end{aligned}
$$
Write
$$A_{l-1} \cup\left(2^{l-1}+1+B_{l-1}\right) \cup\left(2^{l}+1+B_{l-1}\right):=E,\quad
B_{l-1} \cup\left(2^{l-1}+1+A_{l-1}\right) \cup\left(2^{l}+1+A_{l-1}\right):=F.$$
If $n \leq 2^{l}$, then clearly we have $R_{E}(n)=R_{F}(n)$. Now we assume that $2^{l}<n \leq 3 \cdot 2^{l-1}$. It follows that
$$
\begin{aligned}
R_{E}(n) &=\left|\left\{\left(c, c^{\prime}\right): c \in A_{l-1}, c^{\prime} \in 2^{l-1}+1+B_{l-1}, c+c^{\prime}=n\right\}\right| \\
&+\left|\left\{\left(c, c^{\prime}\right): c \in A_{l-1}, c^{\prime} \in 2^{l}+1+B_{l-1}, c+c^{\prime}=n\right\}\right| \\
&+\left|\left\{\left(c, c^{\prime}\right): c<c^{\prime}, c, c^{\prime} \in 2^{l-1}+1+B_{l-1}, c+c^{\prime}=n\right\}\right| \\
&=|\{(c, c^{\prime}): c \in A_{l-1}, c^{\prime} \in B_{l-1}, c+c^{\prime}=n-(2^{l-1}+1)\}| \\
&+\mid\{(c, c^{\prime}): c \in A_{l-1}, c^{\prime} \in B_{l-1}, c+c^{\prime}=n-(2^{l}+1) \mid+R_{B_{l-1}}(n-2(2^{l-1}+1))\end{aligned},$$
and
$$\begin{aligned} R_{F}(n)=&\left|\left\{\left(d, d^{\prime}\right): d \in B_{l-1}, d^{\prime} \in 2^{l-1}+1+A_{l-1}, d+d^{\prime}=n\right\}\right| \\
&+\left|\left\{\left(d, d^{\prime}\right): d \in B_{l-1}, d^{\prime} \in 2^{l}+1+A_{l-1}, d+d^{\prime}=n\right\}\right| \\
&+\left|\left\{\left(d, d^{\prime}\right): d<d^{\prime}, d, d^{\prime} \in 2^{l-1}+1+B_{l-1}, d+d^{\prime}=n\right\}\right| \\
&=|\{(d, d^{\prime}): d \in B_{l-1}, d^{\prime} \in A_{l-1}, d+d^{\prime}=n-(2^{l-1}+1)\}| \\
&+\mid\{(d, d^{\prime}): d \in B_{l-1}, d^{\prime} \in A_{l-1}, d+d^{\prime}=n-(2^{l}+1) \mid+R_{A_{l-1}}(n-2(2^{l-1}+1)).
\end{aligned}.
$$

By Lemma \ref{lem6}, $R_{A_{l-1}}(m)=R_{B_{l-1}}(m)$ holds for any positive integers $m$, and then $R_{E}(n)=R_{F}(n)$.
By Lemma \ref{lem1}, $C\cap [0,3\cdot 2^{l-1}]=E$, $D\cap [0,3\cdot 2^{l-1}]=F$.

Next we shall prove that $m\ge 3 \cdot 2^{l-1}$ is impossible.

If $m>3 \cdot 2^{l-1}$, then by $0\in C$ we have
$$\begin{aligned}
R_{C}\left(3 \cdot 2^{l-1}+1\right) &=\chi_{C}\left(3 \cdot 2^{l-1}+1\right)+\left|\left\{\left(c, c^{\prime}\right): c \in A_{l-1}, c^{\prime} \in 2^{l}+1+B_{l-1}, c+c^{\prime}=3 \cdot 2^{l-1}+1\right\}\right| \\
&+\left|\left\{\left(c, c^{\prime}\right): c<c^{\prime}, c, c^{\prime} \in 2^{l-1}+1+B_{l-1}, c+c^{\prime}=3 \cdot 2^{l-1}+1\right\}\right| \\
&=\chi_{C}\left(3 \cdot 2^{l-1}+1\right)+\left|\left\{\left(c, c^{\prime}\right): c \in A_{l-1}, c^{\prime} \in B_{l-1},c+c^{\prime}=2^{l-1}\right\}\right|+R_{B_{l-1}}\left(2^{l-1}-1\right)
\end{aligned}$$
Similarly we have
$$R_{D}\left(3 \cdot 2^{l-1}+1\right)=\left|\left\{\left(d, d^{\prime}\right): d \in B_{l-1}, d^{\prime} \in A_{l-1}, d+d^{\prime}=2^{l-1}\right\}\right|+R_{A_{l-1}}\left(2^{l-1}-1\right).$$
By Lemma \ref{lem6}, $R_{A_{l-1}}(2^{l-1}-1)=R_{B_{l-1}}(2^{l-1}-1)$. Hence
$R_{C}\left(3 \cdot 2^{l-1}+1\right)=R_{D}\left(3 \cdot 2^{l-1}+1\right)$ if and only if
$\chi_{C}\left(3 \cdot 2^{l-1}+1\right)=0$, that is, $3 \cdot 2^{l-1}+1 \in D$.

Since $1,3 \cdot 2^{l-1}+1,2^{l}+1,2^{l-1}+1 \in D$ and
$$3 \cdot 2^{l-1}+2=1+\left(1+3 \cdot 2^{l-1}\right)=\left(2^{l-1}+1\right)+\left(2^{l}+1\right),$$
it follows that
$$
\begin{aligned}
&R_{D}\left(3 \cdot 2^{l-1}+2\right)=1+1+\left|\left\{\left(d, d^{\prime}\right): d \in B_{l-1}, d^{\prime} \in 2^{l}+1+A_{l-1}, d+d^{\prime}=3 \cdot 2^{l-1}+2\right\}\right| \\
&+\left|\left\{\left(d, d^{\prime}\right): d<d^{\prime}, d, d^{\prime} \in 2^{l-1}+1+A_{l-1}, d+d^{\prime}=3 \cdot 2^{l-1}+2\right\}\right|\\
&=1+1+\mid\left\{\left(d, d^{\prime}\right): d \in B_{l-1}, d^{\prime} \in A_{l-1}, d+d^{\prime}=2^{l-1}+1 \mid+R_{A_{l-1}}(2^{l-1})\right.,\end{aligned}$$
Similarly,
$$
\begin{aligned}
&R_{C}\left(3 \cdot 2^{l-1}+2\right)=\chi_{C}\left(3 \cdot 2^{l-1}+2\right)+\left|\left\{\left(c, c^{\prime}\right): c \in A_{l-1}, c^{\prime} \in B_{l-1}, c+c^{\prime}=2^{l-1}+1\right\}\right|+R_{B_{l-1}}(2^{l-1}).
\end{aligned}
$$
Hence $R_{C}\left(3 \cdot 2^{l-1}+2\right) \neq R_{D}\left(3 \cdot 2^{l-1}+2\right)$, a contradiction.

If $m=3 \cdot 2^{l-1}$, by $3 \cdot 2^{l-1}+2=(2^{l-1}+1)+(2^{l}+1)$, then we have
$$\begin{aligned}
&R_{C} (3 \cdot 2^{l-1}+2 )=\mid \{ (c, c^{\prime} ): c \in A_{l-1}, c^{\prime} \in B_{l-1}, c<c^{\prime}, c+c^{\prime}=2^{l-1}+1 \mid+R_{B_{l-1}} (2^{l-1} ) , \\
&R_{D} (3 \cdot 2^{l-1}+2 )=1+\mid \{ (d, d^{\prime} ): d \in B_{l-1}, d^{\prime} \in A_{l-1}, d<d^{\prime}, d+d^{\prime}=2^{l-1}+1 \mid+R_{A_{l-1}}(2^{l-1}).
\end{aligned}$$

Hence $R_{C}\left(3 \cdot 2^{l-1}+2\right) \neq R_{D}\left(3 \cdot 2^{l-1}+2\right)$, a contradiction.

Therefore, $m<3 \cdot 2^{l-1}$.

Now we have
$$
\begin{aligned}
&C=A_{l-1} \cup\left(2^{l-1}+1+B_{l-1}\right) \cup\left(2^{l}+1+\left(B_{l-1} \cap\left[0, m-\left(2^{l}+1\right)\right]\right)\right), \\
&D=B_{l-1} \cup\left(2^{l-1}+1+A_{l-1}\right) \cup\left(2^{l}+1+\left(A_{l-1} \cap\left[0, m-\left(2^{l}+1\right)\right]\right)\right).
\end{aligned}
$$

Next we prove $m=2^{l}$. Suppose that $m>2^{l}$.

If $m-\left(2^{l}+1\right) \neq 2^{k}-1$, where $k$ is a nonnegative integer.
By Lemma \ref{lem7}, there exists an integer $u$ satisfying $m-\left(2^{l}+1\right)<u<2\left[m-\left(2^{l}+1\right)\right]$ such that
\begin{eqnarray}\label{eq12}R_{A \cap\left[0, m-\left(2^{l}+1\right)\right]}(u)=R_{A_{l-1} \cap\left[0, m-\left(2^{l}+1\right)\right]}(u) \neq R_{B_{l-1} \cap\left[0, m-\left(2^{l}+1\right)\right]}(u)=R_{B \cap\left[0, m-\left(2^{l}+1\right)\right]}(u).\end{eqnarray}

Since $m+(2^{l}+1)<u+2(2^{l}+1)<2 m$, it follows that
\begin{eqnarray*}&&R_{C}(2(2^{l}+1)+u)=R_{2^{l}+1+(B_{l-1} \cap [0, m-(2^{l}+1)])}
(2(2^{l}+1)+u)\\
&=&R_{B_{l-1} \cap [0, m-(2^{l}+1)])}(u)=R_{B \cap [0, m-(2^{l}+1)]}(u).
\end{eqnarray*}
Similarly
\begin{eqnarray*}&&R_{D}(2(2^{l}+1)+u)=R_{2^{l}+1+(A_{l-1} \cap [0, m-(2^{l}+1)])}
(2(2^{l}+1)+u)\\
&=&R_{A_{l-1} \cap [0, m-(2^{l}+1)])}(u)=R_{A \cap [0, m-(2^{l}+1)]}(u).
\end{eqnarray*}

Since $R_{C}\left(2\left(2^{l}+1\right)+u\right)=R_{D}\left(2\left(2^{l}+1\right)+u\right)$,
it follows that $R_{A \cap\left[0, m-\left(2^{l}+1\right)\right]}(u)=R_{B \cap\left[0, m-\left(2^{l}+1\right)\right]}(u)$, a contradiction with \eqref{eq12}.

Now we suppose that $m-(2^l+1)=2^k-1$. Noting that $m<3\cdot 2^{l-1}$, we have
$k<l-1$. Hence
$$
\begin{aligned}
&C=A_{l-1} \cup (2^{l-1}+1+B_{l-1} ) \cup (2^{l}+1+B_{k} ) \\
&D=B_{l-1} \cup (2^{l-1}+1+A_{l-1} ) \cup (2^{l}+1+A_{k} ).
\end{aligned}
$$

If $k=0$, then $2^{l}+1 \in D$. Now we divide the following two cases.

{\bf Case 1.} $2^{l-1}-1\in A_{l-1}$. That is, $2^l\in D$. Hence $2^{l+1}+1=(2^l+1)+2^l$,
and so $R_{C}(2^{l+1}+1)=0$, $R_{D}(2^{l+1}+1)=1$, a contradiction.

{\bf Case 2.} $2^{l-1}-1 \in B_{l-1}$. That is, $2^{l} \in C$. Hence $2^{l-1}-2 \in A_{l-1}$, and so $2^{l}-1 \in D$. Since $2^{l+1}=(2^{l}+1)+(2^{l}-1)$, it follows that $R_{C}(2^{l+1})=0$, $R_{D}(2^{l+1})=1$, a contradiction.

Therefore, we assume that $k>0$. Let $C^{\prime}=2^{l}+2^{k}-C$, $D^{\prime}=2^{l}+2^{k}-D$.

{\bf Case 1.} $2\nmid l-1$, $2\nmid k$. It follows that
$$
\begin{aligned}
&C^{\prime}=A_{k} \cup (2^{k}+A_{l-1}) \cup (2^{l-1}+2^{k}+1+B_{l-1}),\\
&D^{\prime}=B_{k} \cup (2^{k}+B_{l-1}) \cup (2^{l-1}+2^{k}+1+A_{l-1}).
\end{aligned}
$$

{\bf Case 2.} $2\nmid l-1$, $2\mid k$. Then
$$
\begin{aligned}
&C^{\prime}=B_{k} \cup (2^{k}+A_{l-1}) \cup (2^{l-1}+2^{k}+1+B_{l-1}),\\
&D^{\prime}=A_{k} \cup (2^{k}+B_{l-1}) \cup (2^{l-1}+2^{k}+1+A_{l-1}).
\end{aligned}
$$

{\bf Case 3.} $2\mid l-1$, $2\nmid k$. Then
$$
\begin{aligned}
&C^{\prime}=A_{k} \cup (2^{k}+B_{l-1} ) \cup (2^{l-1}+2^{k}+1+A_{l-1} ) \\
&D^{\prime}=B_{k} \cup (2^{k}+A_{l-1} ) \cup (2^{l-1}+2^{k}+1+B_{l-1} ).
\end{aligned}
$$

{\bf Case 4.} $2\mid l-1$, $2\mid k$. Then
$$
\begin{aligned}
&C^{\prime}=B_{k} \cup\left(2^{k}+B_{l-1}\right) \cup\left(2^{l-1}+2^{k}+1+A_{l-1}\right),\\
&D^{\prime}=A_{k} \cup\left(2^{k}+A_{l-1}\right) \cup\left(2^{l-1}+2^{k}+1+B_{l-1}\right).
\end{aligned}
$$

Hence $C^{\prime} \cup D^{\prime}= [0,2^{l}+2^{k} ]\setminus \{2^{l-1}+2^{k} \}$,
$C^{\prime} \cap D^{\prime}=\emptyset$. By Lemma \ref{lem2}, it follows that
$R_{C^{\prime}}(n)=R_{D^{\prime}}(n)$ for all positive integers $n$. By Lemma \ref{lem1},
it suffices to consider the following two cases.

(I) $C^{\prime} \cap  [0,2^{l-1}+2^{k}-1 ]=A \cap [0,2^{l-1}+2^{k}-1 ]$ and
$D^{\prime} \cap  [0,2^{l-1}+2^{k}-1 ]=B \cap [0,2^{l-1}+2^{k}-1 ]$.

(II) $C^{\prime} \cap [0,2^{l-1}+2^{k}-1 ]=B \cap [0,2^{l-1}+2^{k}-1 ]$
and $D^{\prime} \cap [0,2^{l-1}+2^{k}-1 ]=A \cap [0,2^{l-1}+2^{k}-1 ]$.

If Case 1 and Case (I) hold, then
$C^{\prime} \cap\left[0,2^{k}-1\right]=A \cap\left[0,2^{k}-1\right]=A_{k}.$
By the property of $A$, we have
$C^{\prime} \cap [2^{k}, 2^{k+1}-1]=2^{k}+B_{k}.$
By Case 1, we have $C^{\prime} \cap\left[2^{k}, 2^{k+1}-1\right]=2^{k}+A_{k}$, a contradiction.

If Case 1 and Case (II) hold, then
$C^{\prime} \cap\left[0,2^{k}-1\right]=B \cap\left[0,2^{k}-1\right]=B_{k}.$
By Case 1, we have $C^{\prime} \cap\left[0,2^{k}-1\right]=A_{k}$, a contradiction.

If Case 2 and Case (I) hold, then
$C^{\prime} \cap\left[0,2^{k}-1\right]=A_{k}$.
By Case 2, we have $C^{\prime} \cap\left[0,2^{k}-1\right]=B_{k}$, a contradiction.

If Case 2 and Case (II) hold, then $k<l-1$, and so $2^{k}+1<2^{l-1}-1$.
By the property of $A$, it follows that $2^{k}+1 \in A_{l-1}$, and then $2^{k}+\left(2^{k}+1\right) \in 2^{k}+A_{l-1} \subseteq C^{\prime}$. Hence $2^{k}+\left(2^{k}+1\right) \in C^{\prime}$.
By Case (II), $2^{k}+\left(2^{k}+1\right) \in B$. But $2^{k}+$ $\left(2^{k}+1\right)=2^{k+1}+1 \in A$, a contradiction.

If Case 3 and Case (I) hold, then by $2^{k}+1 \in A_{l-1}$, we have $2^{k}+\left(2^{k}+1\right) \in 2^{k}+A_{l-1} \subseteq D^{\prime}$.
By Case (I), $2^{k}+\left(2^{k}+1\right) \in B$. However, by the property of $A$, $2^{k}+\left(2^{k}+1\right)=2^{k+1}+1 \in A$, a contradiction.

If Case 3 and Case (II) hold, then $C^{\prime} \cap [0,2^{k}-1]=B \cap [0,2^{k}-1]=B_{k}$.
By Case 3, $C^{\prime} \cap [0,2^{k}-1 ]=A_{k}$, a contradiction.

If Case 4 and Case (I) hold, then $C^{\prime} \cap\left[0,2^{k}-1\right]=A_{k}$.
By Case 4, $C^{\prime} \cap\left[0,2^{k}-1\right]=B_{k}$, a contradiction.

If Case 4 and Case (II) hold, then $C^{\prime} \cap\left[0,2^{k}-1\right]=B \cap\left[0,2^{k}-1\right]=B_{k}.$
By the property of $B$,
$C^{\prime} \cap\left[2^{k}, 2^{k+1}-1\right]=2^{k}+A_{k}.$
By Case 4, $C^{\prime} \cap\left[2^{k}, 2^{k+1}-1\right]=2^{k}+B_{k}$, a contradiction.

Therefore, $m=2r$.

\end{proof}

%


\end{document}